\documentclass{amsart}
\usepackage{amsfonts, amsbsy, amsmath, amssymb}

\usepackage[figuresright]{rotating} 

\newtheorem{thm}{Theorem}[section]
\newtheorem{lem}[thm]{Lemma}

\newtheorem{conj}[thm]{Conjecture}

\newtheorem{rmk}[thm]{Remark}
\numberwithin{equation}{section}

\theoremstyle{definition}

\begin{document}

\title{A Note Regarding Permutation Binomials over $ \mathbb{F}_{q^2}$}

\author{Stephen D. Lappano}
\address{Department of Mathematics and Statistics,
University of South Florida, Tampa, FL 33620}
\email{slappano@mail.usf.edu}

\begin{abstract}
Let $f=ax+x^{r(q-1)+1}\in \mathbb{F}_{q^2}^*[x], r\in \{5,7\}.$  We give explicit conditions on the values $(q,a)$ for which $f$ is a permutation polynomials of $\mathbb{F}_{q^2}.$
\end{abstract}

\maketitle

\section{Introduction}
A polynomial $f \in \mathbb{F}_q[x]$ is called a \textit{permutation polynomial} (PP) if the mapping $\texttt{x} \mapsto f(\texttt{x})$ induces a permutation on $ \mathbb{F}_q$.  Permutation polynomials taking simple algebraic forms are particularly interesting. While permutation monomials are quite obvious, the situation for permutation binomials is not as well understood. As a result much research has been concerned with finding, classifying and understanding permutation binomials.  The main results of the present paper are the following theorems.
\begin{thm}
Let $f(x)=ax+x^{5q-4}\in\mathbb{F}_{q^{2}}[x]$. Then $f$
is a PP of $\mathbb{F}_{q^{2}}$ iff one of the following occurs:
\begin{enumerate}
\item $q=2^{4k+2}$ and $a^{\frac{q+1}{5}}\neq1$ is a fifth root of unity
\item $q=3^{2}$ and $a^{2}$ is a root of $(1+x) (1+x^2)(2+x+x^2)(2+2 x+x^2) (1+x+x^2+x^4) (1+x^2+x^3+x^4) (1+2 x+x^2+2 x^3+x^4)$
\item $q=19$ and $a^{4}$ is a root of $(1+x)(2+x)(3+x)(4+x)(5+x)(9+x)(10+x)(13+x)(17+x)(16+3x+x^{2})(1+4x+x^{2})(6+18x+x^{2})$
\item $q=29$ and $a^{6}\in\{15,18,22,23\}$
\item $q=7^{2}$ and $a^{10}$ is a root of $(1+4x+x^{2})$
\item $q=59$ and $a^{12}$ is a root of $(4+x)(55+x)(x^{2}+36)$
\item $q=2^{6}$ and $a^{13}$ is a root of $(1+x+x^2)(1+x+x^3)$\end{enumerate}
\end{thm}

\begin{thm}
Let $f(x)=ax+x^{7q-6}\in\mathbb{F}_{q^{2}}[x]$. Then $f$
is a PP of $\mathbb{F}_{q^{2}}$ iff one of the following occurs:
\begin{enumerate}
\item $q=13$ and $a^2$ is a root of $(1+x) (2+x) (3+x) (4+x) (5+x) (6+x) (7+x) (8+x) (9+x) (10+x) (11+x) (12+x+x^2) (9+2 x+x^2) (10+3 x+x^2) (9+4 x+x^2) (12+4 x+x^2) (10+5 x+x^2) (3+6 x+x^2) (1+7 x+x^2) (4+7 x+x^2) (1+8 x+x^2) (12+9 x+x^2) (1+10 x+x^2) (3+12 x+x^2) (4+12 x+x^2) (12+12 x+x^2)$
\item $q=3^3$ and $a^4$ is a root of $(2+x+x^2+x^3) (1+2 x+x^2+x^3) (1+x+2 x^2+x^3) (2+2 x+2 x^2+x^3) (1+2 x+x^2+2 x^3+x^4+2 x^5+x^6)$
\item $q=41$ and $a^6$ is a root of $(9+x) (10+x) (26+x) (30+x) (32+x) (34+x) (35+x) (37+x) (39+2 x+x^2) (1+14 x+x^2) (20+40 x+x^2)$
\end{enumerate}
\end{thm}

It is worth noting the polynomial $f$ in Theorems 1.1 and 1.2 can be written as $f=xh(x^{q-1})$, where $h(x)=a+x^r$. Zieve \cite{Zieve} considered polynomials of this type and under the assumption $a\in \mathbb{F}_{q^2}$ is a (q+1)st root of unity Theorem 1.1 (1) follows from [5,Cor 5.3]. The approach of this paper is quite different from that of \cite{Zieve}, as no restrictions are placed on the nature of $a \in \mathbb{F}_{q^2}$.  We employ a method similar to that of \cite{Hou1,Hou2} and expand the sum $\sum_{x \in \mathbb{F}_{q^2}}f(x)^s$. Surprisingly, considering only a few values of $s$ give explicit results. In \cite{Hou and Lappano} permutation binomials of the form $f=ax+x^{3q-2}$ were completely characterized. The present paper extends this result to polynomials of the form $f=ax+x^{5q-4}$ and $f=ax+x^{7q-6}$. We conclude with a conjecture regarding the behavior of $f=ax+x^{r(q-1)+1}$ for odd primes $r$.

\section{Computations}
Let $f=ax+x^{5q-4}\in \mathbb{F}_{q^2}^*[x]$ and let $0 \leq \alpha , \beta \leq q-1$.  We compute 

\begin{align} \label{2.1}
\sum_{x\in \mathbb{F}_{q^2}}f(x)^{\alpha +\beta q}
 &= \sum_{x\in \mathbb{F}_{q^2}}(ax+x^{5q-4})^\alpha (a^q x^q + x^{5-4q})^\beta \notag \\
 &= \sum_{x\in \mathbb{F}_{q^2}} \sum_{i,j} \binom{\alpha}{i} (ax)^{\alpha -i}x^{(5q-4)i} \binom{\beta}{j}(a^q x^q )^{(\beta-j)}x^{(5-4q)j} \notag \\ 
 &=a^{\alpha+\beta q} \sum_{i,j} \binom{\alpha}{i} \binom{\beta}{j} a^{-i-jq} \sum_{x\in \mathbb{F}_{q^2}} x^{\alpha + \beta q +5(q-1)(i-j)}.
\end{align}
It is clear the inner sum is $0$ unless $\alpha + \beta q \equiv  (\text{ mod } q-1)$, thus $\alpha +\beta = q-1$. 

With $0 \leq \alpha \leq q-1$ and $\beta = q-1-\alpha$ (2.1) becomes
\begin{equation}-a^{(\alpha+1)(1-q)} \sum_{-\alpha -1 +5(i-j) \equiv 0 (\text{mod} q+1)} \binom{\alpha}{i} \binom{q-1-\alpha}{j} a^{-i-jq}
\end{equation}
Since $0 \leq i \leq q-1$ and $0 \leq \beta \leq q-1-\alpha$ it follows  
\begin{equation}
4\alpha +4 - 5q \leq -\alpha -1 +5(i-j) \leq 4\alpha -1 
\end{equation} 
Define $\Gamma(q,\alpha):= \{n \in (q+1)\mathbb{Z} : 4\alpha +4 - 5q \leq n \leq  4\alpha -1 \}$.  Now we have 
\begin{equation}
\sum_{x\in \mathbb{F}_{q^2}}f(x)^{\alpha +\beta q} = -a^{(\alpha +1)(1-q)}\Lambda(q,\alpha,a)
\end{equation}
where 
\begin{equation}
\Lambda(q,\alpha,a)= \sum_{-\alpha -1 +5(i-j)\in \Gamma(q,\alpha)} \binom{\alpha}{i} \binom{q-1-\alpha}{j} a^{-i-jq}
\end{equation}
Now Hermite's criterion implies $f$ is a PP of $\mathbb{F}_{q^2}$ if and only if $0$ is the only root of $f$ in $\mathbb{F}_{q^2}$ and
\begin{equation}
\Lambda(q,\alpha,a)=0 \text{ for each } 0 \leq \alpha \leq q-1.
\end{equation}

\begin{rmk}
Notice if $q+1\equiv 0 \pmod{5},$ then $0$ is the only root of $f$ if and only if $a^{\frac{q+1}{5}} \ne 1.$
\end{rmk}

\begin{lem}
If $f$ is a PP of $\mathbb{F}_{q^2}$, then $q+1 \equiv 0 \text{ mod } 5.$
\end{lem}

\begin{proof}
Assume $f$ is a PP and $q \geq 5$.  First suppose $5 \leq q <8$. Note that $$ \Gamma(q,0)=\{-3(q+1),-2(q+1),-(q+1)\}.$$ By (2.6) we have $$ 0 = \Lambda(5,0,a)= \binom{5-1}{1}a^{-5}=-a^{-5}$$ and $$ 0 = \Lambda(7,0,a)= \binom{7-1}{1}a^{-21}=-a^{-21}.$$  In either case we have a contradiction.  Now suppose $q \geq 8$.  In this case notice $$\Gamma(q,0)=\{ -4(q+1), -3(q+1), -2(q+1), -(q+1) \}.$$ So again by (2.6) we have $$0=\Lambda(q,0,a)=\sum_{k=1}^{4} \binom{q-1}{ \frac{k(q+1)-1}{5}}^{*} a^{-(\frac{k(q+1)-1}{5})q}$$ where $$ \binom{n}{m}^{*} = \begin{cases} \binom{n}{m} &\mbox{ if } m \in \mathbb{Z} \\ 0 &\mbox{ otherwise }
\end{cases}.$$  Now if $(q+1) \not \equiv 0 \pmod{5}$ then exactly one of $\binom{q-1}{ \frac{k(q+1)-1}{5}}^{*}$ is nonzero which contradicts (2.6).  Thus if $f$ is a PP of $\mathbb{F}_{q^2}$ we must have $q+1 \equiv 0 \pmod{5}$.
\end{proof}

\begin{rmk}
Assume $q+1 \equiv 0 \pmod{5} \text{ and } \alpha>0.$  The previous lemma together with (2.5) imply the sum $\Lambda(q,\alpha,0)$ is empty unless $\alpha+1 \equiv 0 \pmod{5}.$
\end{rmk}

\begin{lem}
Assume $q+1 \equiv 0 \pmod{5}$, $\alpha >0$, $\alpha + 1 \equiv 0 \pmod{5}$ and $q \geq 4 \alpha +8$. Set $v= a^{-\frac{q+1}{5}}$, then $$\Lambda(q,\alpha,a)= (-a)^{\frac{\alpha +1}{5}q} \sum_{i=0}^{\alpha}(-1)^i \binom{\alpha}{i} \sum_{l=0}^{4} \binom{i+\frac{4 \alpha - 1 + l}{5}}{\alpha} v^{5i + l}.$$
\end{lem}

\begin{proof}
Since $q \geq 4\alpha +8$ we have $$\Gamma(q,\alpha)=\{-4(q+1),-3(q+1),-2(q+1),-(q+1),0\}.$$  Using (2.5) we see 
\begin{equation*} 
\begin{split}
\Lambda(q,\alpha,a) &= \sum_{-\alpha -1 +5(i-j) \in\Gamma(q,\alpha)} \binom{\alpha}{i} \binom{q-1-\alpha}{j} a^{-i-jq}\\  
&= \sum_{i=0}^{\alpha} \binom{\alpha}{i} \sum_{l=0}^{4} \binom{-1-\alpha}{\frac{1}{5}(l(q+1)-\alpha-1)+i}a^{-i-[\frac{1}{5}(l(q+1)-\alpha-1)+i]q}\\  
&=(-a)^{\frac{q+1}{5}q}\sum_{i=0}^\alpha \binom{\alpha}{i} \sum_{l=0}^4 (-1)^{i}  \binom{\frac{1}{5}(l(q+1)-\alpha-1)+i+\alpha}{\alpha}a^{-\frac{q+1}{5}(l+5i)}\\
&= (-a)^{\frac{\alpha +1}{5}q} \sum_{i=0}^{\alpha}(-1)^i \binom{\alpha}{i} \sum_{l=0}^{4} \binom{i+\frac{4 \alpha - 1 + l}{5}}{\alpha} v^{5i + l}.
\end{split}
\end{equation*}
Between the second and third line we use $\binom{-m}{n}=(-1)^n \binom{n+m-1}{m-1}.$
\end{proof}

\begin{lem}
Assume $q+1 \equiv 0 \pmod{5}$ and $\alpha > 0.$  $\Gamma(q,\alpha)$ contains exactly five consecutive multiples of $q+1$ unless $\alpha=\frac{q-1}{2} \in \mathbb{Z}$ or $\alpha=\frac{q-3}{4}\in \mathbb{Z}.$
\end{lem} 

\begin{proof}
We may assume $\alpha +1 \equiv 0 \pmod{5}.$  Since $\Gamma(q,\alpha)$ in contained in the interval $[4\alpha+4-5q, 4\alpha -1]$ which has length $5(q-1)$ we must have $ 4 \leq | \Gamma(q,\alpha)| \leq 5.$

Suppose  $ 4 =| \Gamma(q,\alpha)|.$  Choose $k$ so $\{(k-3)(q+1), (k-2)(q+1), (k-1)(q+1), k(q+1) \} = \Gamma(q,\alpha).$ Note $q \geq 6$ and $\alpha >0$ force $k \in \{0,1,2,3 \}.$  We have the following inequalities
\begin{equation}
\begin{cases}
4 \alpha -1 \leq k(q+1) +q \\
(k-3)(q+1)-q \leq 4 \alpha +4 -5q.
\end{cases}
\end{equation}
Since $4 \alpha -1, 4\alpha +4 -5q \equiv 0 \pmod{5}$ it follows 
\begin{equation}
\begin{cases}
4 \alpha -1 \leq k(q+1) +q - 4 \\
(k-3)(q+1)-q + 4 \leq 4 \alpha +4 -5q.
\end{cases}
\end{equation}
Since $a \leq q-1$ and $0\leq k \leq 3$, taking the difference of the inequalities in (2.8) reveals $k=0$ which gives $\alpha = \frac{q-3}{4}$ or $k=1$ which gives $\alpha = \frac{q-1}{2}$.
\end{proof}

\begin{lem}
 Assume $q+1 \equiv 0 \pmod{5}$ and $y:=a^{\frac{q+1}{5}} \neq 1$ is a 5th root of unity.  Then for $1 \leq \alpha \leq q-1$ we have
 $$\Lambda(q, \alpha, a) = \begin{cases}
-a^{-\frac{1}{5}(\alpha+1)}(y^{-1}+1+y+y^2) &\mbox{ if } \alpha=\frac{q-1}{2}, \hspace{2mm} \alpha \in \mathbb{Z}\\
-a^{-\frac{1}{5}(\alpha+1)}(1+y+y^2+y^3) &\mbox{ if } \alpha=\frac{q-3}{4}, \hspace{2mm} \alpha \in \mathbb{Z}\\
0 &\mbox{ otherwise. }
\end{cases}$$
\begin{proof}
We may assume $\alpha +1 \equiv 0 \pmod{5}.$  First suppose $\alpha \ne \frac{q-3}{4}, \frac{q-1}{2}.$ Let $K$ be a set of five consecutive integers such that $K(q+1) = \Gamma(q,\alpha).$  Now we have
\begin{equation} 
\begin{split}
 \Lambda(q,\alpha,a) &= \sum_{-\alpha -1 +5(i-j) \in \Gamma(q,\alpha)} \binom{\alpha}{i} \binom{q-1-\alpha}{j} a^{-i-jq}\\   
&= \sum_{k\in K} \sum_{-\alpha-1+5(i-j)=k(q+1)}\binom{\alpha}{i} \binom{q-1-\alpha}{j} a^{-i+j}\\
&= \sum_{k\in K} a^{-\frac{1}{5}[\alpha+1+k(q+1)]} \sum_{i-j=\frac{1}{5}[\alpha+1+k(q+1)]} \binom{\alpha}{\alpha - i} \binom{q-1-\alpha}{j}\\
&= a^{-\frac{1}{5}(\alpha+1)} \sum_{k \in K}y^{-k} \sum_{\alpha-i+j=\frac{1}{5}[4 \alpha -1 - k(q+1)]} \binom{\alpha}{\alpha-i} \binom{q-1-\alpha}{j}\\
&= a^{-\frac{1}{5}(\alpha+1)} \sum_{k \in K}y^{-k} \binom{q-1}{\frac{1}{5}[4 \alpha - 1 - k(q+1)]}\\
&= -a^{-\frac{1}{5}(\alpha+1)} \sum_{k \in K}y^{-k}=0.
\end{split}
\end{equation}

Now suppose $\alpha= \frac{q-1}{2}.$  By the above computation and the previous lemma we have $K =\{-2,-1,0,1\}$, thus
\begin{equation}
\Lambda(q,\alpha,a)= -a^{-\frac{1}{5}(\alpha+1)}(y^{-1}+1+y+y^2).
\end{equation}
Similarly if $\alpha= \frac{q-3}{4},$ we have
\begin{equation}
\Lambda(q,\alpha,a) =-a^{-\frac{1}{5}(\alpha+1)}(1+y+y^2+y^3).
\end{equation}
\end{proof}
\end{lem}

\section{Proof of the Theorem}

\begin{proof}(Theorem1.1)\\
$(\Leftarrow)$  Cases (ii)-(vii) are easily verified by a computer.  Assume (i), that is $q=2^{4k+2}$ and $ a^{ \frac{q+1}{5}} \neq 1$ is a fifth root of unity.  Lemma 2.6 gives $\Lambda(q,\alpha,a)=0$ for each $0 \leq \alpha \leq q-1$ and $0$ is the only root of $f$, so $f$ is a PP by (2.6)\\
 \\
$(\Rightarrow)$  Assume $f$ is a PP.  By Lemma (2.2) we have $q+1 \equiv 0 \pmod{5}.$  Let $y:=a^{\frac{q+1}{5}}.$  If $y \ne 1$ is a fifth root of unity then Lemma 2.6 implies $q$ must be even.  Thus $q=2^{4k+2}$ and we have case (i).

Now suppose $1+y+y^2+y^3+y^4 \neq 0$.  The sum in the RHS of Lemma 2.4 is a polynomial in $v(=y^{-1})$ and can be easily computed for small values of $\alpha$ with the help of a computer algebra system.  For a few values of $\alpha$ we find \begin{equation}
\Lambda(q,\alpha,a) = (-a)^\frac{\alpha+1}{5}v(1+v+v^2+v^3+v^4) 
\begin{cases}
5^{-4}g_4 (v) &\mbox{ if } \alpha = 4, q\geq 24\\
5^{-10}g_9 (v) &\mbox{ if } \alpha = 9, q\geq 44\\
5^{-16}g_{14} (v) &\mbox{ if } \alpha = 14, q\geq 64\\
5^{-28}g_{24} (v) &\mbox{ if } \alpha = 24, q\geq 104.\\
\end{cases}
\end{equation} The polynomials $g_\alpha$ are given in the appendix. Write $R(p_1,p_2)$ for the resultant of polynomials $p_1, p_2.$  Then
\begin{equation}
GCD(R(g_4,g_9),R(g_4,g_{14}))=2^{15}3^{3}5^{197}.
\end{equation}
Thus if $q\geq 64$ we must have $p \hspace{1mm} (=char \mathbb{F}_{q^2}) \in \{2,3\}.$  Since $q+1 \equiv 0 \pmod{5}$ there are only a few prime powers $q<64$ with $p\hspace{1mm} (=char \mathbb{F}_{q^2})\neq 2,3.$
\begin{itemize}
\item When $q=19$, a computer search results in case (iii) 
\item When $q=29$, a computer search results in case (iv) 
\item When $q=49$, a computer search results in case (v) 
\item When $q=59$, a computer search results in case (vi)
\end{itemize}

When $p=2$ we have $GCD(g_4,g_{24})=x$ thus $q<104.$  Since $q+1 \equiv 0 \pmod{5}$ and $q>4$ we only need to consider $q=64.$  A computer search results in case (vii).

When $p=3$  we have $GCD(g_4,g_9)=1$ thus $q<44.$  Again since $q+1 \equiv 0 \pmod{5}$ we only need to consider $q=9.$  A computer search results in case (ii).
\end{proof}
Theorem 1.2 is proved using a similar method which leads us to the following conjecture.
\begin{conj}
Let $ r>2$ be a fixed prime.  If both $ (q+1) \equiv 0 \pmod{r}$ and $a^{\frac{q+1}{r}}$ is not an $r-th$ root of unity; we conjecture there are only finitely many values $(q,a)$ for which $f=ax+x^{r(q-1)+1}\in \mathbb{F}_{q^2}^*[x]$ is a permutation polynomial of $\mathbb{F}_{q^2}.$
\end{conj}
\section{Appendix}
\begin{equation*} 
\begin{split} 
g_4(x) &=g_4(x)=44+75 x+115 x^2+165 x^3-2899 x^4-1152 x^5-1465 x^6-1825 x^7\\
&-2235 x^8+25427 x^9+4122 x^{10}+4805 x^{11}+5555 x^{12}+6375 x^{13}-58357 x^{14}\\
&-3514 x^{15}-3915 x^{16}-4345 x^{17}-4805 x^{18}+38454 x^{19}.
\end{split}
\end{equation*}
\begin{equation*} 
\begin{split}
g_9(x) &=35464+47564 x+41954 x^2-3211 x^3-121771 x^4+3307018 x^5+7237538 x^6\\
&+13764443 x^7+24034088 x^8-399905587 x^9-247096418 x^{10}-374180158 x^{11}\\
&-550629013 x^{12}-791033308 x^{13}+10166063897 x^{14}+3488769646 x^{15}\\
&+4732145066 x^{16}+6330108851 x^{17}+8362735316 x^{18}-90589540129 x^{19}\\
&-20005926320 x^{20}-25545256270 x^{21}-32342528845 x^{22}-40628485270 x^{23}\\
&+389225321705 x^{24}+56878393066 x^{25}+69794856416 x^{26}+85151720951 x^{27}\\
&+103326084566 x^{28}-901674492499 x^{29}-84876624338 x^{30}-101203850008 x^{31}\\
&-120170621113 x^{32}-142129134058 x^{33}+1152208354517 x^{34}+63621134698 x^{35}\\
&+74215593188 x^{36}+86304925343 x^{37}+100064515838 x^{38}-764098747192 x^{39}\\
&-18879570941 x^{40}-21644171461 x^{41}-24754360696 x^{42}-28246292086 x^{43}\\
&+205243145184 x^{44}.
\end{split}
\end{equation*}
\begin{equation*} 
\begin{split}
g_{14}(x) &=35781192+45103176 x+34642700 x^2-11075651 x^3-104451417 x^4+3431018744 x^5\\&+5425819552 x^6+5815099340 x^7+1030592053 x^8-15702529689 x^9+367570246196 x^{10}\\&+915262570648 x^{11}+1948805856320 x^{12}+3779801961517 x^{13}-62553432822181 x^{14}\\&-47505978162672 x^{15}-79121737743896 x^{16}-127716748824160 x^{17}-200786722216499 x^{18}\\&+2746238364681602 x^{19}+1263694841543105 x^{20}+1868009627371040 x^{21}\\&+2718322826664950 x^{22}+3900065855839735 x^{23}-46407807995168830 x^{24}\\&-15130762640643679 x^{25}-20922156001611712 x^{26}-28643147196518350 x^{27}\\&-38852104359262613 x^{28}+415138746369911354 x^{29}+101011800565346642 x^{30}\\&+133546181274464111 x^{31}+175316634655627970 x^{32}+228624694439538979 x^{33}\\&-2240965131247477702 x^{34}-414819831279200662 x^{35}-530625761486025481 x^{36}\\&-675157922777947390 x^{37}-854715990669328349 x^{38}+7803518358751564382 x^{39}\\&+1100711546227581974 x^{40}+1372388078602423607 x^{41}+1703988861163039670 x^{42}\\&+2107260628357829383 x^{43}-18124791008882124634 x^{44}-1923722972889132730 x^{45}\\&-2349566656784986915 x^{46}-2860023585150150200 x^{47}-3470096113600539985 x^{48}\\&+28364072170221684830 x^{49}+2200015572073495862 x^{50}+2641585104385038536 x^{51}\\&+3162990481363018525 x^{52}+3777155092039529639 x^{53}-29542409091657957562 x^{54}\\&-1584053771764262386 x^{55}-1874825156457447288 x^{56}-2213821365876774435 x^{57}\\&-2608223671181041657 x^{58}+19628014114205307016 x^{59}+651404631123094716 x^{60}\\&+761528771717834208 x^{61}+888518897148324570 x^{62}+1034700084660340082 x^{63}\\&-7525847208868343576 x^{64}-116633198562800052 x^{65}-134898641118903336 x^{66}\\&-155761841308153260 x^{67}-179556106888260384 x^{68}+1266995051549992032 x^{69}
\end{split}
\end{equation*}

\begin{equation*} 
\begin{split}
g_{24}(x) &=44199864566676+53277891868890 x+36951353492365 x^2-18447292166275 x^3\\&-115981817761656 x^4+5683172059809972 x^5+7671280571862570 x^6\\&+6327550580716885 x^7-1472857823809435 x^8-18209145388579992 x^9\\&+496676536774208388 x^{10}+771501576163548690 x^{11}+774677809879397185 x^{12}\\&+37151453118769745 x^{13}-2080007375935924008 x^{14}+44904298634276465124 x^{15}\\&+85192441936752637650 x^{16}+111451370786904498265 x^{17}+47809116140932255865 x^{18}\\&-289357227498865856904 x^{19}+6805964762978995053180 x^{20}+19701780140886237080850 x^{21}\\&+47920556078520020855875 x^{22}+105337499516914088443775 x^{23}\\&-1763163188957551294558680 x^{24}-1682149911521612848440924 x^{25}\\&-3144091106960055505926510 x^{26}-5682560762926174819291685 x^{27}\\&-9983113995337864503302425 x^{28}+145844209029690653966024044 x^{29}\\&+90577048940932553337666992 x^{30}+149016337368570701236013970 x^{31}\\&+240987977179133628538814635 x^{32}+383707733869425939573259415 x^{33}\\&-5054637186099365312920130012 x^{34}-2383073542636350764433826072 x^{35}\\&-3646247878744859791470683310 x^{36}-5516927650609039918496642165 x^{37}\\&-8260762058311010597210260105 x^{38}+99947418327478671197343833527 x^{39}\\&+37844995423440695266939743959 x^{40}+55100153059398963201126006300 x^{41}\\&+79573062038413273868028870415 x^{42}+114034864385838660589245973115 x^{43}\\&-1283855920027521585601121843789 x^{44}-401631430484182503011941004665 x^{45}\\&-563418559621522038915547815300 x^{46}-785422361095522275273631000625 x^{47}\\&-1088343656606667747839725560325 x^{48}+11515350760355635501320142255915 x^{49}\\&+3023936138319755361442261262591 x^{50}+4119412652650088403766862240940 x^{51}\\&+5583398973821401437427782781615 x^{52}+7530907539377792043231905759675 x^{53}\\&-75471967365836479567646155794821 x^{54}-16787688185166611656225020168433 x^{55}\\&-22327203688191441773802872732580 x^{56}-29570337359172614836200572834465 x^{57}\\&-39004838361833906826391115026885 x^{58}+372587302842602423374479932324863 x^{59}\\&+70520166819123145016529674257763 x^{60}+91921490099273043258578272890040 x^{61}\\&+119393417476195244286408801766635 x^{62}+154543277170360802146843655221595 x^{63}\\&-1414460450943061212979844779136033 x^{64}-228058586275344486040909532074621 x^{65}\\&-292198163240869437036528542283800 x^{66}-373236999859930139181490117490085 x^{67}\\&-475340219041052709214076298517285 x^{68}+4186641921387463113636090427865791 x^{69}\\&+573982399597440241104412076252915 x^{70}+724508808520889474869872176906800 x^{71}\\&+912087128194582180606507235839875 x^{72}+1145264181936242472433040520458075 x^{73}\\&-9742873878315091411250160134457665 x^{74}-1130413280103442753936876665890989 x^{75}\\&-1408264364840845831044139437784160 x^{76}-1750319535256268238132497892063885 x^{77}\end{split}
\end{equation*}
\begin{equation*} 
\begin{split}
&-2170502922925974506143566661090725 x^{78}+17890834639608180012805054289642259 x^{79}\\&+1742739271574108271273793805711627 x^{80}+2145966682431740273469435944163595 x^{81}\\&+2637022907200101153997087624591860 x^{82}+3233886164548639605146540321869515 x^{83}\\&-25897969665493387445366349883836597 x^{84}-2092770412365513238435254665605237 x^{85}\\&-2550276778698834519482774414234285x^{86}-3102050153955665578994040465961140 x^{87}\\&-3766364889683287116775343845510605 x^{88}+29374000062448380447889073547561267 x^{89}\\&+1935240033819795981039415092100619 x^{90}+2336265403685016505132971195469975 x^{91}\\&+2815697619541879892767440305985240 x^{92}+3387974339410094717815325956323015 x^{93}\\&-25785924105952569948928004112378849 x^{94}-1350431582795782203758136667410805 x^{95}\\&-1616445075520877395405741712969225 x^{96}-1931945507456689143942634660617000 x^{97}\\&-2305611169359156003531542005529025 x^{98}+17156418696304763600550408562151055 x^{99}\\&+687379082077441716165690906155771 x^{100}+816415958977468475853193747835215 x^{101}\\&+968348850584524348447504741362840 x^{102}+1147010610713807759039043095096625 x^{103}\\&-8358195150530508420416077021700451 x^{104}-240714899339982433449408005104533 x^{105}\\&-283875024040741105674985093276305 x^{106}-334355514938579365555961289992040 x^{107}\\&-393328645903551948398787580245735 x^{108}+2810835580815912962362589234243613 x^{109}\\&+51827832972624672481156065143283 x^{110}+60721533591007894432573966947615 x^{111}\\&+71059972468668725537659588986360 x^{112}+83064891280671737260332843826245 x^{113}\\&-582904099186926848986624808653503 x^{114}-5170193326947891817908520747581 x^{115}\\&-6020861165457984420843719642625 x^{116}-7004114740388167029285613941960 x^{117}\\&-8139491600243587212407368112835 x^{118}+56154582678607837122596101351251 x^{119}\end{split}
\end{equation*}

\end{document}